\documentclass[12pt,reqno]{amsart}

\usepackage{amsmath,amsfonts,amsthm,amssymb,amsrefs,esint,enumerate,color}

\usepackage[mathscr]{eucal}
\usepackage[normalem]{ulem}
\usepackage[margin=2cm]{geometry}
\usepackage{bbm}
\usepackage{tikz}
 
\theoremstyle{plain}
\newtheorem{theorem}{Theorem}[section]
\newtheorem{cor}[theorem]{Corollary}
\newtheorem{lemma}[theorem]{Lemma}
\newtheorem{prop}[theorem]{Proposition}

\theoremstyle{remark}

\newtheorem*{remarka}{Remark}

\newcounter{countD}

\newcommand{\R}{\mathbb{R}}

\renewcommand{\d}{\mathrm{d}}
\newcommand{\av}{{\mathrm{av}}}
\newcommand{\diff}{{\mathfrak{D}}}

\newcommand{\dev}{{\mathrm{dev}}}
\newcommand{\fin}{{\mathrm{fin}}}
\DeclareMathOperator*{\essinf}{\mathrm{essinf}}

\def\intslash{\rlap{\kern  .32em $\mspace {.5mu}\backslash$ }\int}
\def\qsl{{\rlap{\kern  .32em $\mspace {.5mu}\backslash$ }\int_{Q_x}}}

\def\diam{{\text{\rm diam}}}

\def\cf{\emph{cf.\ }}

\def\diam{{\text{\rm diam}}}

\def\rad{{\text{\rm rad}}}

\def\noi{\noindent}

\def\lc{\lesssim}

\def\ga{\gamma}

\def\eps{\varepsilon}

\def\la{\lambda}

\def\Om{\Omega}

\def\fD{{\mathfrak {D}}}

\def\fS{{\mathfrak {S}}}

\def\bbN{{\mathbb {N}}}

\def\bbR{{\mathbb {R}}}
\def\bbS{{\mathbb {S}}}

\def\cA{{\mathcal {A}}}
\def\cB{{\mathcal {B}}}
\def\cC{{\mathcal {C}}}

\def\cE{{\mathcal {E}}}

\def\cI{{\mathcal {I}}}

\def\cK{{\mathcal {K}}}

\def\cR{{\mathcal {R}}}

\def\cU{{\mathcal {U}}}

\def\tempp{{p}}
\def\tempq{{q}}

\def\sC{{\mathscr {C}}}

\title{Two weight estimates for difference quotients}

\dedicatory{Dedicated to Neil Trudinger on the occasion of his 85th birthday}

\author{Carlos P\'erez, Andreas Seeger, Brian Street, Po-Lam Yung}
\address{C. P\'erez, Department of Mathematics \\ University of the Basque Country, Ikerbasque (Basque Foundation for Science) and  BCAM -- Basque Center for Applied Mathematics, Bilbao, Spain}
\email{cperez@bcamath.org}

\address{A. Seeger, Department of Mathematics \\ 
	University of Wisconsin, Madison \\
	480 Lincoln Drive, Madison, WI, 53706, USA}
\email{seeger@math.wisc.edu}

\address{B. Street, Department of Mathematics \\ 
	University of Wisconsin, Madison \\
	480 Lincoln Drive, Madison, WI, 53706, USA}
\email{street@math.wisc.edu}

\address{P.-L. Yung, Mathematical Sciences Institute \\
	Australian National University \\
	Canberra ACT 2601 \\
	Australia \\
	CNRS-ANU International Research Laboratory FAMSI \\ (France Australia Mathematical Sciences and Interactions) \\
	and 
	Department of Mathematics \\
	The Chinese University of Hong Kong\\
	Shatin, Hong Kong} 
\email{PoLam.Yung@anu.edu.au} \email{plyung@math.cuhk.edu.hk}

\begin{document}
	\begin{abstract}
		We prove local and global  two weight estimates  in which we bound   difference quotients of a function in terms of certain weighted $L^p$ norms of its gradient.
	\end{abstract}
	\maketitle
	\section{Introduction}

	In the study of partial differential equations, the weak differentiability of solutions may often be deduced through a consideration of their difference quotients. The classic monograph \cite{03561752} by Gilbarg and Trudinger already highlights the following basic facts: 
	\begin{itemize}
		\item The first order difference quotients of a $W^{1,p}$ function are uniformly in $L^p$ (see Lemma 7.23 in \cite{03561752});
		\item Conversely, if the first order difference quotients of an $L^p$ function are uniformly in $L^p$, then $u \in W^{1,p}$ (see Lemma 7.24 in \cite{03561752}).
	\end{itemize}

	More recently,  other interesting and useful characterizations of membership of $\nabla f$ in $L^p$ have been found which involve first order differences, or  difference quotients. 
	Let us define  for
	$x\neq y$ \begin{equation}\label{Qb-def} \diff_{b}f(x,y)  =  \frac{|f(y)-f(x)|}{|y-x|^{1+b}}
	\end{equation} and  set $\diff_b f(x,x)=0$. Let $b=\gamma/p$. 
	One might try  to  measure the $L^p$ norm of $(x,y)\mapsto \fD_{\ga/p} f(x,y)$ with respect to the measure $|y-x|^{\gamma-d} \d x\d y$ which  reduces   to the 
	$\gamma$-independent 
	expression
	$$\Big(\iint_{(x,y)\in\R^d \times \R^d} \frac{|f(x)-f(y)|^p}{|x-y|^{d+p}}\d x\d y \Big)^{1/p}.
	$$ 
	However, it is well known that if this integral is finite for a locally integrable function $f$, then $f$ must be constant almost everywhere \cite{brezisconstants}.   The failure of any such $L^p$ estimate  can  be remedied by replacing the $L^p(|x-y|^{\ga-d} \d x\d y)$ norm   with a  corresponding Marcinkiewicz ($L^{p,\infty}$) quasinorm.
	In particular, 
	for $1<p<\infty$, $\gamma\neq 0$ and $f\in C^\infty_{\mathrm {c}}(\R^d)$ 
	the $L^{p,\infty}$ quasi-norm  of $\diff_{\ga/p}f$ 
	in $\R^d \times \R^d$ with respect to the measure $|y-x|^{\gamma-d} \d x \d y$ is bounded by $C\|\nabla f\|_{L^p(\R^d)}$, i.e.
	\begin{equation} \label{eq:BSVY}
		\sup_{\lambda > 0} \lambda^p \iint\limits _{\substack{(x,y)\in\R^d \times \R^d :\\\diff_{\ga/p}f(x,y) > \lambda}} |y-x|^{\gamma-d} \d x \d y \lesssim \|\nabla f\|_{L^p(\R^d)}^p.
	\end{equation}
	For the case  $\gamma=d$, in which  one considers Lebesgue measure on $\bbR^d\times \bbR^d$,  \eqref{eq:BSVY}  was shown by  Brezis, Van Schaftingen and Yung \cite{BVY} who also covered the case $p=1$.  For the case $\gamma=-p$ (in  $ (-\infty,-1)$),  in which one considers differences instead of difference quotients, \eqref{eq:BSVY} had already been  shown by Nguyen \cite{Nguyen06} (and independently by A. Ponce and J. Van Schaftingen).
	For other choices of $\gamma$  it was shown by Brezis, Van Schaftingen and two of the authors \cite{BSVY} that  inequality \eqref{eq:BSVY} holds for $1<p<\infty$, $\gamma\neq 0$  and for $p=1$ and $\gamma\in \bbR\setminus[-1,0]$, and these  $\gamma$-ranges are sharp.  The results can be extended to functions  in the homogeneous Sobolev space  
	$\dot{W}^{1,p}(\R^d)$, and to functions of bounded variation when $p=1$; moreover for $f\in \dot W^{1,p}$ one gets equivalences of quasi-norms (see \cite{BSVY}) and thus the desired characterizations of $\dot W^{1,p}$. 
	The results for $\gamma=d$ were used in \cite{2021Lorentz} to extend  the scope of Gagliardo--Nirenberg inequalities in \cite{BrezisMironescu18}. 
	Further developments can be found in \cites{BourgainNguyen,MR4458224, DSSVY, MR4525737,MR4512396,MR4249777,Yang_weighted,Nguyen08,MR4645709,MR4777237}, and  the surveys  \cites{BSVY_revisited, nguyensurvey25}.
	
	The purpose of this paper is to extend the weak type inequality \eqref{eq:BSVY} to versions with possibly different weight functions on either side of the inequality,  so called  {\it two-weight inequalities}.
	We also cover local versions where $\bbR^d$ may be  replaced by a suitable  open subset $\Omega\subset \bbR^d$.
	
	Given a nonnegative measurable function $u$ on $\Omega$ we define a measure $\mu_\gamma^u$ on $\Omega\times \Omega$ 
	by 
	\begin{equation}\label{eq:mugamma}  \d \mu^u_{\gamma} (x,y)  =u(x) |y-x|^{\gamma-d} \d x \d y,
	\end{equation} 
	on $\Omega\times\Omega$, with  $u$  a non-negative measurable function on $\Omega$. 
	Even though $u(x) |y-x|^{\gamma-d}$ may not be locally integrable,
	\eqref{eq:mugamma} defines an inner regular  measure %
	\cite[Proposition 412Q]{FremlinVol4} (i.e.
	$\mu^u_\gamma(E)=\sup_{K\subset E} \mu^u_\gamma (K)$ where the supremum is taken over all compact subsets of $E$).

	For  $f \in C^1(\Om)$ we are interested  in  a priori $\mu^u_\gamma$ bounds for the super-level sets
	$\{(x,y)\in \Om\times \Om: \diff_{\ga/p} f(x,y)>\la\}$  where $\nabla f$ is measured in a weighted $L^p(w)$ norm with respect to a nonnegative weight $w$.

	We shall assume that the weight pair  $(u,w)$  belongs to the Muckenhoupt class  $A_p(\Om)$. 
	Recall that for $1<p<\infty$, a pair of non-negative  measurable  functions $(u,w)$ on $\R^d$ is said to be in  $A_p(\Om)$, written $(u,w)\in A_p(\Om)$, if
	\begin{equation}\label{Ap-2w-Om}
		[u,w]_{A_p(\Om)}:=\sup_B \frac{1}{|B|} \int_{B\cap\Om}  u(x) \d x \,\, \Big(\frac{1}{|B|} \int_{B\cap\Om}  w(y)^{-p'/p}\d y\Big)^{p/p'}
		<\infty,
	\end{equation}
	with  the supremum  taken over all balls $B$ in $\R^d$.
	For $p=1$  a pair of weights $(u,w)$ belongs to the class $A_1(\Om)$ if there is a constant $K$ such that for any ball $B \subset \R^d$,
	\begin{equation}\label{eq:A11-Om}
		\frac{1}{ |B| } \int_{B\cap\Om} u(x) \d x 
		\le K \essinf_{y\in B\cap\Om} w(y). 
	\end{equation}
	We  denote the infimum of all possible  constants $K$ by $[u,w]_{A_1}$. For $\Om=\bbR^d$ we recover the standard Muckenhoupt conditions for pairs of weights in $\bbR^d$. 
	Membership of $(u,w) $ in $A_p(\Om)$ implies that $u\in L^1_{
		\mathrm{loc}}(\Om)$ is locally integrable and that $w(y)>0$ for almost every $y\in \Om$; for this and other basic facts we refer to \S IV.1 and \S IV.5 in  \cite{GCRdF} (\cf also \S\ref{sec:two-weight} below). 
	
	We will study our problem in  appropriate {\it convex} domains $\Omega$,  and with weight pairs $(u,w)\in A_p(\Om)$.  In the class of  admissible domains we wish to include $\bbR^d$ itself, as well as balls and cubes in $\bbR^d$; however domains which are long and thin  may pose problems. To formulate this quantitatively we  denote for a domain $\Omega\subset \bbR^d$ the diameter by   
	$\diam(\Om)=\sup_{x,y\in \Omega}|x-y|$,  and define the  {\it ball deviation} of $\Om$ as 
	\[ \dev(\Om)= \sup_{x\in \Om}\,\,\sup_{0 < r\leq\diam(\Om)} \frac{|B(x,r)|}{|B(x,r)\cap \Om|}. \]
	In two dimensions it may be interpreted as the square of an  eccentricity of $\Om$.
	Note that $\dev(\Om)\ge 1$ for all domains and  $\dev(\bbR^d)=1$. Balls and cubes have  ball deviation  uniformly bounded by a dimensional constant.  All bounded convex domains in $\bbR^d$  have finite ball deviation.
	Among unbounded domains, open convex cones have finite ball deviation while  infinite strips of finite width do not have finite ball deviation.

	We will  prove a priori estimates for 
	the difference quotient operator $\diff_b$
	when acting on functions in $C^1(\Om)$. By completion the a priori estimates extend to larger classes of functions  but we will not   explicitly address 
	the question of what happens for a general function in a weighted Sobolev space with weight $w$. For $b=\ga/p$ we want to estimate the measure  $\mu_{\ga}^u$ of the  sets
	\begin{equation*} \label{sublevel-sets} E_{\la,b } (f)= \{(x,y)\in \Om\times \Om: \diff_{b} f(x,y)>\la\}.
	\end{equation*}

	\begin{theorem} \label{thm:main}
		Let $d \geq 1$, $\Om \subset \R^d$ be a convex domain with finite ball deviation and $f \in C^1(\Om)$. Assume  $1 \leq \tempq \leq \tempp < \infty$, $\tempp > d (\tempq-1)$ and $(u,w) \in A_\tempq(\Om)$. If $\gamma\in \bbR\setminus [-qd,0]$ then 
		for all $\la>0$,
		\begin{equation*}
			\mu_\gamma^u(E_{\la,\ga /p } (f))\lc_{d,\gamma, \tempp, \tempq} [u,w]_{A_\tempq(\Om) } \Big( \frac{\dev(\Om)}{\la}\Big)^\tempp \int_\Om |\nabla f(y)|^\tempp w(y)\d y\,.
		\end{equation*}
		The implicit constant depends only on $d, \ga, p, q$  and is independent of $\Om$, the weights $u$ and $w$, and the function $f$.
	\end{theorem}
	
	By specializing to the case $q = p$, we obtain:
	\begin{cor}\label{cor:Ap}
		Let $d \geq 1$, $\Om \subset \R^d$ be a convex domain with finite ball deviation and $f \in C^1(\Om)$. If $\gamma \in \bbR\setminus [-pd,0]$ and $1 \leq p < \frac{d}{d-1}$  (in particular $1\le p<\infty$ when $d=1$) 
		then 
		for all $(u,w) \in A_p(\Om)$ and all $\la>0$, 
		\begin{equation}\label{eq:p-p}
			\mu_\gamma^u(E_{\la,\ga/ p } (f))\lc_{d,p,\gamma} [u,w]_{A_p(\Om) } \Big(\frac{\dev(\Om)}{\la}\Big)^p \int_\Om |\nabla f(y)|^p w(y)\d y. 
		\end{equation}
	\end{cor}
	For the case  $q = 1$ in Theorem \ref{thm:main} we obtain:
	\begin{cor}\label{cor:A1}
		Let $d \geq 1$, $\Om \subset \R^d$ be a convex domain with finite ball deviation and $f \in  C^1(\Om)$. Let 
		$1 \leq p < \infty$ and 
		$(u,w) \in A_1(\Om)$.
		If  $\gamma \in \R \setminus [-d,0]$, 
		then  for all $\la>0$,
		\begin{equation*}
			\mu_\gamma^u(E_{\la, \ga/ p} (f))\lc_{d,p,\gamma} [u,w]_{A_1(\Om) }\Big(\frac{\dev(\Om)} {\la}\Big)^p\int_\Om |\nabla f(y)|^p w(y)\d y. 
		\end{equation*}
		In particular, for any non-negative locally integrable  function $u$ on $\Om$, we have
		\begin{equation*}
			\mu_\gamma^u(E_{\la, \ga/p} (f))\lc_{d,p,\gamma} \Big(\frac{\dev(\Om) }{\la}\Big)^p \int_\Om |\nabla f(y)|^p M_\Om u(y)\d y. 
		\end{equation*}
	\end{cor}
	The second statement holds since the pair $(u, M_\Om u)$ belongs to the class $A_1(\Om)$. 
	It would be interesting to find new necessary conditions for the ranges of $p$ and  $\gamma\in [-qd,0]$. Some further perspectives on this are offered in Section~\ref{sect:speculations} below.

	Weighted analogs of inequality \eqref{eq:BSVY} with one weight $u=w$  were  considered by Li--Yang--Yuan--Zhang--Zhao \cite{Yang_weighted}. Their results  rely on the wavelet analysis ideas by 
	Cohen--Dahmen--Daubechies--DeVore \cite{Cohen-et-al} and raise questions about  possible improvements in the range of negative $\gamma$ in Corollary \ref{cor:Ap}.  We note that  the example  in \S\ref{sec:example} indicates  a fundamental difference between the one-weight and two-weight theories for our difference quotients.

	\subsection*{\it Notation} 
	Denote by $|E|$ the Lebesgue measure of a set $E$ in $\R^d$ and by $\chi_E$ its characteristic function. For a set $E$ of positive finite Lebesgue measure let  
	\[\av_E f=\frac{1}{|E|}\int_E f \d x\] be the average of $f$ over $E$  with respect to the Lebesgue measure. Sometimes we write $u(E)$ for the integral $\int_E u$.
	A ball centered at $x$ and of radius $r$ is denoted by $B(x,r)$. 
	For $t > 0$, the dilation of a ball $B$ by $t$ times around its center will be denoted $tB$.  Let  $L^{p,\infty}(X,\mu)$ denote the Marcinkiewicz (or weak type $p$) space on the measure space $(X,\mu)$ which consists of all measurable  $F$ on $X$ such that 
	\[
	[F]_{L^{p,\infty}(X,\mu)} ^p:= \sup_{\lambda > 0} \lambda^p\, \mu\big(\{x \in X \colon |F(x)| > \lambda\}\big)
	\] is finite. $F\mapsto [F]_{L^{p,\infty}}$ is a quasi-norm which is homogeneous of degree $1$. Typically we will consider $X=\Om\times\Omega$ with $\mu$ given by $\mu_\ga^u$ as in 
	\eqref{eq:mugamma}.
	We also use the notation $L^{p,\infty}(v)$ if $v$ is a weight function and $\d\mu= v(x) \d x$ with $\d x$ Lebesgue measure. Similarly, for the usual $L^p$ spaces  we concurrently use the notations  $L^p(v)$ and  $L^p(\mu)$  if $\d \mu=v\d x$. 
	The side length of a cube $Q$ will be denoted $\ell(Q)$.

	\section{Main estimates}

	\subsection{Preliminaries on two weight $A_p$ conditions} 
	\label{sec:two-weight} 
	
	The significance of the $A_p(\Om)$ class arises from a  well-known result by Muckenhoupt  for the $\Omega$-variant $M_\Om$ of the Hardy-Littlewood maximal operator,  given by 
	\begin{equation} \label{eq:HLOm} M_\Om f(x)=\sup_{x\in B} \frac{1}{|B|} \int_{B\cap \Om} |f(y)|\d y. \end{equation} 
	Namely  we have $[u,w]_{A_p(\Om)}<\infty$  if and only if the operator $M_\Om$ maps $L^p(\Om,w)$ to $L^{p,\infty}(\Om, u)$, and  then
	\begin{equation}\label{charactAp-two-weights}
		\|M_\Om\|_{L^p(\Om, w) \to L^{p,\infty}(\Om, u)} \approx  [u,w]^{1/p}_{A_p}. 
	\end{equation}
	In other words
	\[ \sup_{\lambda>0}\la^p \int_{\{x: M_\Omega f(x) >\la \}} u(x) \d x\lc C^p \int |f(y)|^p w(y)\d y\] and the infimum over admissible  $C$ in this inequality is comparable to $[u,w]_{A_p}^{1/p}$. 
	The implicit constants in \eqref{charactAp-two-weights}
	depend on the dimension. There is  also  an analogous result  for the {\it centered version}
	$$M_\Om ^{\mathrm{cent}}f(x)=\sup_{r>0} \frac{1}{|B(x,r)|} \int_{B(x,r)\cap \Om} |f(y)|\d y;$$  
	moreover we have  analogues where in both versions balls are replaced by cubes. 
	
	From 
	\cite{GCRdF},  the $A_p(\Om) $ class of weights 
	is characterized by the condition 
	\begin{equation}\label{Ap-2wchart1}
		\Big(\frac{1}{|B|}  \int_{B\cap \Om}  f(y)  \d y\Big)^{p} u(B\cap \Omega)\leq C
		\int_{B\cap \Om} f(y)^p w(y)   \d y
	\end{equation}
	for any ball $B$ and nonnegative measurable $f$; and  the best constant $C$ is equal to  $[u,w]_{A_p(\Om)}$.
	For $p = 1$ this characterization  is immediate. For $p > 1$ one direction follows from H\"older's inequality.
	\begin{multline*}
		\Big(\frac{1}{|B|}  \int_{B\cap\Om} f(y) \d y\Big)^{p} u(B\cap\Om) \\ \leq \Big(  \int_{B\cap\Om} f(y)^p w(y)  \d y\Big) \Big(\frac{1}{|B|} \int_{B \cap\Om} w(y)^{-p'/p}  \d y \Big)^{p/p'} \Big( \frac{1}{|B|} \int_{B\cap \Om} u(x)  \d x \Big)\,,
	\end{multline*}
	and the other one follows  by testing \eqref{Ap-2wchart1} with  $f=w^{-p'/p}$. 
	
	A useful form of this inequality says that for any $t \geq 1$ and any ball $B$, we have
	\begin{align}\label{Ap-2wchart_local} \notag
		\Big(\frac{1}{|B|}  \int_{B\cap\Om} f \, \d y\Big)^{p} u((t B)\cap\Om)
		&\le   t^{pd} \Big(\frac{1}{|t B|}  \int_{(t B)\cap \Om} f \chi_B \, \d y\Big)^{p} u(t B \cap \Om) \\&\leq t^{p d} [u,w]_{A_p} 
		\int_{B\cap\Om} f^p w\, \d y.
	\end{align}
	To see this, we just apply \eqref{Ap-2wchart1} to $t B$ and $f \chi_B$ in place of $B$ and $f$.

	\subsection{The main technical ingredients}\label{sec:tech}
	We provide   the  main technical ingredients in the proofs of Theorem \ref{thm:main}. In this section we assume that $\Omega$ is  a {\it  general measurable  subset of $ \bbR^d$ with positive Lebesgue measure.} Define, for $b \in \R$, an operator $T_b$ acting on measurable  nonnegative functions $g$ by 
	\begin{equation}\label{eq:Tgammadef} T_{b} g(x,r) =  \frac{1}{r^{d + b}} \int_{B(x,r) \cap \Om} g(z) \d z, \quad x\in \Om, r>0.
	\end{equation} 
	We estimate this for the parameter $b=\gamma/p$ and choose as the measure on $\Om \times (0,\infty)$
	\[\d \nu_\gamma^u = u(x) r^{\ga-1} \d x\d r\,.\]
	As $\mu^u_\gamma$ above,  $\nu^u_\gamma$ is an inner regular measure,  for all $\gamma\in \bbR$ \cite[Proposition 412Q]{FremlinVol4}.

	\begin{theorem} \label{thm:g_weakLp}
		Let $d \geq 1$, and let $\Omega\subset \bbR^d$ have positive  Lebesgue  measure.  Let $1 \le  \tempp < \infty$, and let $g$ be a non-negative measurable function on $\Om$.
		\begin{enumerate}[(a)]
			\item \label{thm:g_weakLp_a} If $(u,w) \in A_p(\Om)$, then for any $\gamma > 0$, we have
			\[
			\left[ T_{\gamma/\tempp} g \right]_{L^{\tempp,\infty}(\Om \times (0,\infty), \nu_\ga^u)} \lesssim_{d, \gamma, p} [u,w]_{A_p(\Om)}^{1/\tempp} \|g\|_{L^\tempp(\Om, w)}.
			\]
			\item \label{thm:g_weakLp_b}
			If $1\le q\le p$ and $(u,w) \in A_\tempq(\Om)$, then for any $\gamma < -\tempq d$ we have
			\[
			\left[ T_{\gamma/\tempp} g \right]_{L^{\tempp,\infty}(\Om \times (0,\infty), \nu_\ga^u)} \lesssim_{d,\gamma,p,q} [u,w]_{A_\tempq(\Om)}^{1/\tempp} \|g\|_{L^\tempp(\Om, w)}.
			\]
		\end{enumerate}
	\end{theorem}

	\begin{proof}
		For $\lambda > 0$, let 
		\begin{equation*}
			E(\lambda) = \big\{(x,r) \in \Om \times (0,\infty) \colon T_{\ga/\tempp} g(x,r) > \lambda \big \}.
		\end{equation*} 
		
		\noi{\it The case $\gamma>0$.} 
		By the inner regularity of $\nu^u_\gamma$ it suffices to prove that for any compact subset $\cK$ of $E(\lambda)$, we have           
		\begin{equation}\label{eq:Tgoal}
			\nu_\gamma^u(\cK) \lc_{d, \ga, p} [u,w]_{A_p(\Om)} \la^{-\tempp} \int_\Om |g(y)|^\tempp w(y)  \d y.
		\end{equation}

		The set $\{B(x,r) \times (0,2r) \colon (x,r) \in \cK\}$ covers $\cK$, and hence admits a finite subcover, say
		\[ B(x_i,r_i) \times (0,2r_i), \text{  for  $i=1,\dots,N$.} \]
		Let $\sC_{\mathrm{fin}}$ be the collection of open balls $B(x_i,r_i)$, $i=1,\dots,N$. 
		Since there are only finitely many balls in $\sC_{\mathrm{fin}}$, each of which has a finite radius,
		the Vitali covering selection algorithm (see e.g. \cite{00447275} Chapter I, Section 3)
		gives  a disjoint subcollection $\mathcal{I}'$ of $\sC_{\mathrm{fin}}$ such that every $B \in \sC_{\mathrm{fin}}$ is contained in $5B'$ for some $B' \in \mathcal{I}'$, with $r_B \leq 2 r_{B'}$ (we write $r_B$ for the radius of $B$). Now if $(x,r) \in \cK$, then $(x,r) \in B(x_i,r_i) \times (0,2r_i)$ for some $i$, with $B(x_i,r_i) \in \sC_{\mathrm{fin}}$, so together with $x \in \Om$ we obtain $(x,r) \in (B \cap \Om) \times [0,2r_B]$ for some $B \in \sC_{\mathrm{fin}}$. Using this  together with the covering property of $\mathcal{I}'$, we obtain
		\[
		\cK \subset \bigcup_{B \in \sC_{\mathrm{fin}}} (B \cap \Om) \times [0,2r_B] \subset \bigcup_{B' \in \mathcal{I}'} (5B' \cap \Om) \times [0,4r_{B'}].
		\]
		(In short, if a ball $B$ is contained in a dilation of another ball $B'$, then the Carleson tent $B \times [0, 2r_B]$ over $B$ is contained in the dilation of the Carleson tent $B' \times [0, 2r_{B'}]$ over $B'$.) 
		It follows that 
		\begin{align*}
			\int_{\cK} u(x) r^{\gamma-1} \d x \d r 
			&\leq \sum_{B' \in \mathcal{I}'} \int_{5B' \cap \Om} u(x) \d x \int_0^{4 r_{B'}} r^{\gamma-1} \d r \\
			&\lesssim_{\ga} \sum_{B' \in \mathcal{I}'} r_{B'}^{\gamma}  \, u(5B' \cap \Om).
		\end{align*}
		We used $\ga > 0$ in the last inequality. Now observe  $r_{B'}^{\gamma} \approx_{d,p} \Big( \frac{1}{|B'|} r_{B'}^{d + \frac{\gamma}{\tempp}}\Big)^\tempp$ and using that $B' \in \mathcal{I}'$, we obtain
		\begin{align*}
			&\int_{\cK} u(x) r^{\gamma-1} \d x \d r \lc_{d,\ga,p} \frac{1}{\lambda^\tempp} \sum_{B' \in \mathcal{I}'} \Big( \frac{1}{|B'|} \int_{B' \cap \Om} g(z) \d z \Big)^\tempp u(5B' \cap \Om) 
			\\\notag &  \lesssim_{d,p} \frac{[u,w]_{A_\tempp(\Om)}}{\lambda^\tempp} \sum_{B' \in \mathcal{I}'}  \int_{B' \cap \Om} g(z)^\tempp w(z)\, \d z 
			\lc \frac{[u,w]_{A_\tempp(\Om)}}{\lambda^\tempp} \int_\Om|g(z)|^\tempp w(z) \d z. \notag
		\end{align*}
		Here the second inequality is 
		from \eqref{Ap-2wchart_local}, and the third inequality uses the disjointness of the balls 
		in  $\mathcal{I}'$.  We established  \eqref{eq:Tgoal} in the case $\gamma>0$. 
		
		\noi{\it The case $\gamma<-\tempq d$ with $\tempq \in [1,\tempp]$.} %
		
		Let $\cK\subset E(\la)$ be a compact set. Again, by inner regularity of $\nu^u_\gamma$ it suffices to show 
		that 
		\begin{equation}\label{eq:TgoalK}
			\nu_\gamma^u(\cK)\lc_{d, \ga, p, q} [u,w]_{A_{\tempq}(\Om)} \la^{-\tempp} \int_\Om |g(y)|^\tempp w(y) \d y.
		\end{equation} 
		For every $(x,r)$ in $\cK$ there exists an open  ball $B(x,r)$ centered at $x$ with radius $r$ such that 
		\[\frac{1}{r^{d+\frac\gamma \tempp}}\int_{ B(x,r)\cap \Om }g(z) \d z>\lambda\] 
		The open sets $B(x,r)\times (\frac r2, \infty)$ cover $\cK$ and therefore there is a finite subcover, i.e. there exists a finite family $\sC_0$ of balls $B$ of radius $r_B$ such that  $r_B^{-d-\gamma/\tempp}  \int_{B\cap \Om} g\d z>\lambda$, and  $\cK\subset \bigcup_{B\in \sC_0} Z_B$, with $Z_B=(B \cap \Om) \times (r_B/2, \infty)$.

		We use a reverse Vitali selection procedure.
		We select  disjoint balls   
		$B_1,\dots, B_N$, such that the center of $B_j$ 
		is $x_j$,  its radius is $r_j$, 
		$B_j\in \sC_{j-1}$,  the radius of $B_j$ is the {\it minimum} of all radii of balls in $\sC_{j-1}$, and $\sC_j$ consists of all balls in $\sC_{j-1}$ which are   disjoint from $B_j$
		(by induction these balls have radius at least $r_j$ and these balls are disjoint from all already selected balls $B_1,\dots B_j$). 
		
		Let \[S(B_j)= \Big\{ (x,\rho) \in \Om \times [r_j/2,\infty) \colon  |x-x_{j} | \le 6\rho \Big\}.\]
		
		We claim that if $B$ is any ball in $\sC_{j-1}$ such that $B\cap B_j\neq \emptyset$, then 
		\[Z_B=(B \cap \Om) \times (r_B/2,\infty)\subset S(B_j).\]  
		Indeed, let $x\in B$, $z\in B\cap B_j$ then $|x-x_j|\le |x-z|+|z-x_{j} |\le 2 r_B+r_{B_j} \le 3 r_B$; hence if $\rho\ge r_{B}/2$ then $|x-x_j|\le 3r_B\le 6\rho$ and thus $Z_B \subset S(B_j)$ .
		
		In particular we get $\cK\subset \bigcup_{j=1}^N S(B_j)$ and therefore $\nu_\gamma^u(\cK) \le \sum_{j=1}^N \nu_\gamma^u (S(B_{j})).$
		
		We fix $j$ and let for $n=0,1,2,\dots$
		\[S_n(B_j)= \{ (x,\rho) \in \Om \times \R \colon |x-x_j|\le 6\rho, \,2^{n-1} r_j \le \rho< 2^n r_j\}. \] Then $S(B_j)=\bigcup_{n =0}^\infty S_n(B_j)$.
		
		Now since $B_j \in \sC_0$
		we may use the condition $r_j^{dp+\ga}\le \la^{-p}(\int_{B_j\cap\Om} g)^p$ and estimate
		\begin{align*}
			\nu_\gamma^u (S_n(B_j)) &\le \int_{2^{n-1} r_j} ^{2^n r_j} 
			\rho ^{\gamma-1} \int_{B(x_j,6\rho) \cap \Om} u(x) \,\d x \, \d\rho \\
			&\lc_{\ga} (2^n r_j)^\gamma \int_{B(x_j,6\cdot 2^n r_j) \cap \Om} u(x) \d x \\
			&\lc_{d,\ga,p} 2^{n \gamma} \Big( \lambda^{-1} \frac{1}{|B_j|} \int_{B_j \cap \Om} g(z) \d z \Big)^\tempp u((6 \cdot 2^n B_j) \cap \Om).\end{align*}
		If $\gamma < -\tempq d$, with $\tempq \in [1,\tempp]$, we first apply H\"older's inequality in $g$:
		\begin{equation} \label{eq:gHolder}
			\Big( \frac{1}{|B_j|} \int_{B_j \cap \Om} g(x) \d x \Big)^\tempp
			\leq \Big( \frac{1}{|B_j|} \int_{B_j \cap \Om} g(x)^{\tempp/\tempq} \d x \Big)^\tempq
		\end{equation}
		and then use \eqref{Ap-2wchart_local} with exponent $\tempq$, %
		\begin{equation} \label{Ap-2wchart_local-q}  u((6 \cdot 2^n B_j) \cap \Om) \Big(\frac{1}{|B_j|} \int_{B_j \cap\Omega} g(y)^{p/q} \d y\Big)^q \lc 2^{ndq} [u,w]_{A_q(\Om)} \int_{B_j\cap\Om} g(y)^p w(y) \d y.
		\end{equation}
		We obtain
		\[
		\nu_\gamma^u (S_n(B_j)) \lc 6^q 2^{n (\gamma + \tempq d)} \lambda^{-\tempp} [u,w]_{A_\tempq(\Om)} \int_{B_j \cap \Om} g(y)^\tempp w(y) \d y
		\]
		and, since $\gamma<-\tempq d$,  we may sum in $n$ and get
		$$\nu_\gamma^u(S(B_j)) \lc_{d, \ga , p, q} \la^{-\tempp} [u,w]_{A_\tempq(\Om)} \int_{B_j\cap \Om} g(y)^\tempp w(y) \d y.$$ By the disjointness of the $B_j$, 
		\begin{align*}\nu_\gamma^u(\cK)\leq \sum_{j=1}^N \nu_\gamma^u(S(B_j))\lc \frac{[u,w]_{A_\tempq(\Om)}}{\la^{\tempp}} \sum_{j=1}^N \int_{B_j\cap \Om} g(y)^\tempp w(y)\d y \leq\frac{[u,w]_{A_\tempq(\Om)}}{\la^{\tempp}}  \int_\Om g(y)^\tempp w(y) \d y.
		\end{align*}
		and \eqref{eq:TgoalK} is proved. 
	\end{proof}

	For any $b \in \R$, any nonnegative measurable function $g$ on $\Om$ and any $k\ge 2$ define 
	\begin{equation}\label{eq:Skgammadef} T^k_{b} g(x,y) =  \frac{2^{kd}}{|y-x|^{d + b}} \int_{B(y, 2^{-k}|y-x|) \cap \Om} g(z) \d z. 
	\end{equation} 
	Note that here we are breaking the symmetry in $x,y$ by integrating over the ball of radius $2^{-k}|x-y|$ centered at $y$ while the weight function $u$ depends on $x$. We use the inner regular measure $\mu^u_\gamma$ as defined in \eqref{eq:mugamma}. As before the relevant parameter will be $b=\gamma/p$.

	\begin{theorem} \label{thm:g_weakLp-k} Let $d \geq 1$, $\Omega\subset \bbR^d$ be a measurable subset of positive Lebesgue measure, $k \geq 2$, and $g$ be a non-negative measurable function on $\Om$.  Let $1\le \tempq \le \tempp <\infty$ and $(u,w) \in A_\tempq(\Om)$.
		Then, for any $\gamma \in \bbR\setminus [-qd,0]$, 
		\[\left[ T^k_{\gamma/\tempp} g\right]_{L^{\tempp,\infty}(\Om \times \Om, \mu_\gamma^u)} \lesssim_{d,\gamma,p,q} 2^{\frac{kd(\tempq-1)}{\tempp} } [u,w]_{A_\tempq(\Om)}^{1/\tempp} \|g\|_{L^\tempp(\Om, w)}.
		\]
		The implicit constant is independent of $k$. %
	\end{theorem}
	\begin{proof} 
		Fix $k \geq 2$. For any $x ,y \in \Om$, let $B^k_{x,y} := B(y, 2^{-k} |y - x|)$. 
		For $\lambda > 0$, let 
		\[
		E_k(\lambda) = \big\{(x,y) \in \Om \times \Om \colon x \ne y, T^k_{\gamma/\tempp} g(x,y)>
		\lambda \big \}.
		\]
		We can rewrite this as
		\[
		E_k(\lambda) = \Big\{(x,y) \in \Om \times \Om \colon x \ne y,  2^{k \gamma} \mathrm{rad}(B^k_{x,y})^{\gamma} < \Big( \frac{c_d}{\lambda} \frac{1}{|B^k_{x,y}|} \int_{B^k_{x,y} \cap \Om} g(z) \d z \Big)^\tempp \Big \}
		\]
		where $c_d$ is the volume of the unit ball in $\R^d$. 			
		
		\noi{\it The case $\gamma > 0$.}  Let $\mathcal{I}$ be the collection of 
		balls $B=B^k_{x,y}$ for which  $(x,y) \in E_k(\lambda)$. 
		
		Arguing as in the proof of Theorem \ref{thm:g_weakLp} (this time using a Whitney decomposition of 
		$\cU=\{(x,y)\in \Omega\times\Omega:x\neq y\}$ and  the local integrability of $(x,y)\mapsto |x-y|^{\gamma-d}$ away from the diagonal)  it suffices to show for any  compact set $\cK\subset E_k(\lambda)$ that 
		\begin{equation}\label{eq:K2-negative-ga-estimate} \mu^u_\gamma(\cK)\lc_{d,\ga,p,q} \la^{-p} [u,w]_{A_q(\Om)} \int_{\Om} |g(x)|^p w(x) \, \d x.
		\end{equation}
		Now $(x,y) \in 2^{k+1} B^k_{x,y} \times B^k_{x,y}$, and so if  $(x,y) \in E_k(\lambda)$, then $(x,y) \in ((2^{k+1} B )\cap \Om) \times (B \cap \Om)$ for some $B \in \mathcal{I}$. By the  compactness of $\cK$ there exists a finite collection  $\sC_{\mathrm{fin}}\subset \cI$ such that the open sets $(2^{k+1} B\times B)$ cover $\cK$ and since $\cK\subset \Omega\times \Omega$ we have 
		\[ 
		\cK \subset \bigcup_{B \in \sC_{\mathrm{fin}}} (2^{k+1} B \cap \Om) \times (B \cap \Om). \]            
		By Vitali covering, there exists a disjoint subcollection $\mathcal{I}'$ of $ \sC_{\mathrm{fin}}$  such that every $B \in \sC_{\mathrm {fin}}$ is contained in $3 B'$ for some $B' \in \mathcal{I}'$. Hence
		\[
		\cK
		\subset \bigcup_{B' \in \mathcal{I}'} (6 \cdot 2^k B' \cap \Om) \times 3 B'.
		\] 
		Now for $B' \in \cI'$, we estimate
		\begin{align*}
			\mu_\gamma^u((6 \cdot 2^k B' \cap \Om) \times 3 B' )
			&= \int_{(6 \cdot 2^k B') \cap \Om} u(x) \int_{3 B'} |x-y|^{\gamma - d} \d y \, \d x.
		\end{align*}
		Note that since $\gamma > 0$, for any $x \in \R^d$, 
		\[
		\int_{3 B'} |x-y|^{\gamma - d} \d y \lc_{d,\ga} r_{B'}^{\gamma} \Big( 1 + \frac{|x-c_{B'}|}{r_{B'}} \Big)^{\gamma-d},
		\]
		where $c_{B'}$ is the center of $B'$ and $r_{B'}$ is the radius of $B'$; this can be verified by splitting into the case $x \in 6 B'$ and $x \notin 6 B'$ respectively, and the case $x \in 6 B'$ requires using the condition $\gamma > 0$. As a result,
		\begin{align*}
			\mu_\gamma^u((6 \cdot 2^k B' \cap \Om) \times (3 B' \cap \Om))
			&\lc_{d,\ga} r_{B'}^{\gamma} \int_{(6 \cdot 2^k B') \cap \Om} u(x) \Big( 1 + \frac{|x-c_{B'}|}{r_{B'}} \Big)^{\gamma-d} \d x. 
		\end{align*}
		We decompose
		\[
		6 \cdot 2^k B' \subseteq B' \cup \bigcup_{n=1}^{k+3} (2^n B' \setminus 2^{n-1} B')
		\]
		and use 
		\[
		1 + \frac{|x-c_{B'}|}{r_{B'}} 
		\approx
		\begin{cases}
			1 &\quad \text{if $x \in B'$}\\
			2^n & \quad \text{if $x \in (2^n B') \setminus (2^{n-1} B')$}
		\end{cases}.
		\]
		Since $B'\in \cI$ we have  
		\[
		r_{B'}^{\gamma} \lc_{d,p} 2^{-k \gamma} \lambda^{-\tempp} \Big( \frac{1}{|B'|} \int_{B' \cap \Om} g(x) \d x \Big)^\tempp,
		\]
		and therefore
		\begin{equation*} \label{eq:sum_n}
			\begin{split}
				& \quad r_{B'}^{\gamma} \int_{(6 \cdot 2^k B') \cap \Om} u(x) \Big( 1 + \frac{|x-c_{B'}|}{r_{B'}} \Big)^{\gamma-d} \d x \\
				&\lc_{d,p}  2^{-k \gamma} \lambda^{-\tempp} \Big( \frac{1}{|B'|} \int_{B' \cap \Om} g(x) \d x \Big)^\tempp \sum_{n=0}^{k+3} 2^{n(\gamma-d)}  u((2^n B') \cap \Om)
				\\
				&\lc_{d,p}  2^{-k \gamma} \lambda^{-\tempp} \Big( \frac{1}{|B'|} \int_{B' \cap \Om} g(x)^{p/q} \d x \Big)^q\sum_{n=0}^{k+3} 2^{n(\gamma-d)}  u((2^n B') \cap \Om),
			\end{split}
		\end{equation*}
		with  H\"older's inequality \eqref{eq:gHolder}.
		Using \eqref{Ap-2wchart_local} as in 
		\eqref{Ap-2wchart_local-q}, this shows that
		\begin{align*}
			\mu_\gamma^u((6 \cdot 2^k B' \cap \Om) \times (3 B' \cap \Om))\lc_{d,\ga,p} 
			& \, 2^{-k\gamma} \lambda^{-\tempp} [u,w]_{A_\tempq(\Om)} \sum_{n=0}^{k+3} 2^{n(\gamma+d (\tempq-1))} \int_{B' \cap \Om} g(x)^\tempp w(x) \d x \\
			\lc_{d,\ga,p,q} & \, \lambda^{-\tempp}  2^{k d (\tempq-1)} [u,w]_{A_\tempq(\Om)} \int_{B' \cap \Om} g(x)^\tempp w(x) \d x.
		\end{align*}
		Summing over the disjoint $B' \in \cI'$ we obtain
		\[
		\mu_{\gamma}^u (\cK) \lc_{d,\ga,p,q} \lambda^{-\tempp}  2^{k d (\tempq-1)} [u,w]_{A_\tempq(\Om)} \int_\Om g(x)^\tempp w(x) \d x
		\]
		Taking the supremum over compact $\cK \subset E_k(\lambda)$ gives the desired estimate for $\gamma>0$.

		\noi{\it The case $\gamma<-\tempq d$.} %
		Fix $k \geq 2$. Let
		\[
		E_k(\lambda) = \Big\{ (x,y) \in \Om \times \Om \colon x \ne y, \frac{2^{kd}}{|y-x|^{d+\frac{\gamma}{\tempp}}} \int_{B(y,2^{-k}|y-x|) \cap \Om} g(x) \d x > \lambda \Big\}.
		\]
		Let $\cI_k$ be the collection of open balls $B(y,2^{-k}|y-x|)$ for which $(x,y) \in E_k(\lambda)$. We denote by $c_B$ the center of the ball $B$ and by $r_B$ its radius.
		Then
		\begin{equation}\label{eq:stopk}
			r_B^{\gamma} < \lambda^{-\tempp} 2^{-k \gamma} \Big( \frac{1}{r_B^d} \int_{B \cap \Om} g(x) \d x \Big)^\tempp
		\end{equation}
		for every $B \in \cI_k$, and
		\[
		E_k(\lambda) \subset \bigcup_{B \in \cI_k} (\Om \setminus 2^{k-1} \overline{B}) \times B.
		\]
		For any compact subset $\cK$ of $E_k(\lambda)$, there exists a finite subcollection $\cC_\fin \subset \cI_k$ such that \[ \cK \subset \bigcup_{B \in \cC_\fin} Z_{k,B} \text{ where } Z_{k,B}:= (\Om \setminus 2^{k-1} \overline{B}) \times B.\] As in the proof of Theorem \ref{thm:g_weakLp}  we  use a reverse Vitali selection procedure, selecting a disjoint collection $B_1, \dots, B_N$ from $\cC_\fin$ such that for every $B \in \cC_\fin$, there exists $j \in \{1, \dots, N\}$ for which $B \cap B_j \ne \emptyset$, and $r_B \geq r_{B_j}$. Write
		\[
		\fS_k(B_j) = \big\{(x,y) \in \Om \times \R^d \colon |y-c_{B_j}| \leq 3 \cdot 2^{-(k-2)} |x-y| \text{ and } 2^{-(k-2)} |x-y| \geq r_{B_j}\big\}.
		\]
		If $B \in \cC_\fin$ and $B \cap B_j \ne \emptyset$ with $r_B \ge r_{B_j}$, we claim that $Z_B \subset \fS_k(B_j)$. Indeed, for any $(x,y) \in Z_B$, one has 
		\[
		2^{k-1} r_B < |x-c_B| \leq |x-y| + |y-c_B| \leq |x-y| + r_B
		\]
		which implies, since $k \geq 2$, that $$|x-y| > (2^{k-1} - 1) r_B \geq 2^{k-2} r_B.$$ If $r_B \geq r_{B_j}$, then this implies $|x-y| \geq 2^{k-2} r_{B_j}$, i.e. $2^{-(k-2)} |x-y| \geq r_{B_j}$; also, if we pick a point $z \in B \cap B_j \ne \emptyset$, then %
		\[
		|y-c_{B_j}| \leq |y-z| + |z-c_{B_j}| \leq 2r_B + r_{B_j} \leq 3 r_B \leq 3 \cdot 2^{-(k-2)} |x-y|,
		\]
		proving that $Z_{k,B} \subset \fS_k(B_j)$. It follows that $
		\cK \subset \bigcup_{j=1}^N \fS_k(B_j).
		$ 
		
		We note that for $(x,y)\in \fS_k(B_j)$ we have $|x-y|\ge 2^{k-2}r_{B_j} $. 
		Let 
		\[
		\fS_{0,k}(B_j)=\{(x,y): 2^{k-2} r_{B_j}\le |x-y|\le 2^{k-1} r_{B_j},\, |y-c_{B_j}| \leq 3 \cdot 2^{-(k-2)} |x-y|\} 
		\] and for $n=1,2,\dots$
		\[ \fS_{n,k}(B_j)=\{(x,y): 2^{n+k-2} r_{B_j}\le |x-y|\le 2^{n+k-1} r_{B_j}, 
		|y-c_{B_j}|\le 3\cdot 2^{-(k-2)} |x-y|\} 
		\]
		Then
		\[\fS_k(B_j)\subset \bigcup_{n=0}^\infty \fS_{n,k} (B_j).\]
		
		We estimate $\mu_\gamma^u(\fS_{0,k})$.
		Note that $(x,y) \in \fS_{0,k}$ implies $2^{k-2} r_{B_j}\le |x-y|\le 2^{k-1} r_{B_j}$ and \[|y-c_{B_j}|\le 3\cdot 2^{-(k-2)} |x-y|,\] which implies $|x-c_{B_j}| \le (2^{k-1}+6) r_{B_j} \le 2^{k+1} r_{B_j}$ (since 
		$2^{k+1} - 2^{k-1} = 3 \cdot 2^{k-1} \ge 
		6$ for $k\ge 2$). In other words, $x \in 2^{k+1} B_j$. Also, $|y-c_{B_j}| \leq 3 \cdot 2^{-(k-2)} 2^{k-1} r_{B_j} \leq 6 r_{B_j}$.
		As a result, \[
		\begin{split}
			\mu_{\gamma}^u(\fS_{0,k}(B_j)) 
			&\leq \int\limits_{(2^{k+1}  B_j) \cap \Om} u(x) \int\limits_{\substack{y \in \R^d \colon  |y-c_{B_j}|\le 6 r_{B_j}\\ 2^{k-2}r_{B_j}\le |x-y| \le  2^{k-1} r_{B_j}}} |x-y|^{\gamma - d} \d y \, \d x \\
			&\lc _{d,\ga}  u((2^{k+1}   B_j) \cap \Om) (2^k r_{B_j})^{\gamma-d} r_{B_j}^d\\
			&\lc_{d,\ga,p} 2^{-kd} \lambda^{-\tempp} \Big( \frac{1}{|B_j|} \int_{B_j \cap \Om} g(x) \d x \Big)^\tempp u((2^{k+1}  B_j) \cap \Om) \\
		\end{split}
		\]
		where we have used \eqref{eq:stopk} in the last inequality.  
		From  H\"older's inequality and again \eqref{Ap-2wchart_local} for $q$ instead of $p$ 
		\[\Big( \frac{1}{|B_j|} \int_{B_j \cap \Om} g(x) \d x \Big)^\tempp u((2^{k+1}  B_j) \cap \Om) \lc_{d, q} 2^{kdq} \int_{B_j \cap \Om} g(y)^p w(y) \d y \] and hence summing over the disjoint $B_j$'s we obtain
		\[\sum_{j=1}^N \mu_{\gamma}^u(\fS_{0,k}(B_j)) 
		\lc_{d,\ga,p} 2^{kd (q-1)}  \la^{-p} [u,w]_{A_q(\Om)}\int_\Om g(y)^p w(y) \d y.\]
		
		Similarly for  $n\ge 1$ we observe that
		if  $2^{n+k-2} r_{B_j}\le |x-y|\le 2^{n+k-1} r_{B_j}$ and $|y-c_{B_j}|\le 3\cdot 2^{-(k-2)} |x-y|$ then $|x-c_{B_j}|\le (2^{n+k-1} + 6) r_{B_j} \leq 2^{n+k} r_{B_j}$ and also $|y-c_{B_j}|\le 6 \cdot 2^{n} r_{B_j}.$ Hence, 
		\[
		\begin{split}
			\mu_{\gamma}^u(\fS_{n,k}(B_j)) 
			&\leq \int\limits_{(2^{n+k}  B_j) \cap \Om} u(x) \int\limits_{\substack{y \in \R^d \colon  |y-c_{B_j}|\le 6 \cdot 2^{n} r_{B_j}\\ 2^{n+k-2}r_{B_j}\le |x-y| \le  2^{n+k-1} r_{B_j}}} |x-y|^{\gamma - d} \d y \, \d x \\
			&\lc _{d,\ga}  u((2^{n+k}   B_j) \cap \Om) (2^{n+k} r_{B_j})^{\gamma-d} (2^{n} r_{B_j})^d\\
			&\lc_{d,\ga,p} 2^{n\gamma} 2^{-kd} u((2^{n+k}   B_j) \cap \Om) 
			(2^kr_{B_j})^\gamma
			\\
			&\lc_{d,\ga,p} 2^{n\gamma} 2^{-kd}  u((2^{n+k}   B_j)\cap\Om) 
			\lambda^{-\tempp} \Big( \frac{1}{|B_j|} \int_{B_j \cap \Om} g(x) \d x \Big)^\tempp 
		\end{split}
		\] by  \eqref{eq:stopk}. 
		By H\"older's inequality and \eqref{Ap-2wchart_local},
		\[ u((2^{n+k}   B_j)\cap\Om) 
		\Big( \frac{1}{|B_j|} \int_{B_j \cap \Om} g(x) \d x \Big)^\tempp \leq 2^{(n+k)dq} [u,w]_{A_q(\Om)} \int_{B_j \cap \Omega} g(y)^p w(y)\d y\]
		and therefore 
		\begin{align*}\mu_{\gamma}^u(\fS_{n,k}(B_j)) 
			&\lc_{d,\ga,p,q} 2^{n(dq+\gamma)}  2^{kd (q-1)}\la^{-p} [u,w]_{A_q(\Om)}\int_{B_j \cap \Omega} g(y)^p w(y) \d y. \qedhere
		\end{align*}
		Since $\gamma<-dq$ we may sum in $n$ and then sum over $j = 1, \dots, N$ using disjointness of the $B_j$'s. We then obtain the desired bound for $\mu^u_{\ga}(\cK)$.
	\end{proof}

	\section{Proof of Theorem \ref{thm:main}} \label{sec:proof}

	In this proof we shall now have to assume that $\Omega$ is a convex domain of bounded ball deviation. 
	As a replacement of the fundamental theorem of calculus in one dimension  we shall use a ``subrepresentation formula" for these  domains 
	(see e.g. Lemma 7.16 of \cite{03561752} or more generally Section 8.1 from \cite{AdamsHedberg}, or the proof of equation (4.5) in \cite{Juha}) which is well known to  lead to Poincaré-type inequalities. 
	We refer the reader to \cite{DreDu2008} for extensions and applications of the classical Martio--Reshetnyak pointwise estimate for John domains, which is proven in detail in \cite{Hajlasz2001}.  We also note that there  are  improved versions where  the gradient is replaced by appropriate fractional difference quotients, with the Bourgain--Brezis--Mironescu factor $(1-s)$ in front. These can be found in \cites{HMP2025a, HMP2025b}, with applications to pointwise bounds for maximal rough singular integrals.      For more about Poincar\'e inequalities on metric measure spaces see \cite{heinonenetalbook}.

	For the convenience of the reader we include a proof which tracks the dependency on the ball deviation of $\Om$.

	\begin{lemma}\label{lem:poincare}
		Let $\Omega\subset \bbR^d$ be a convex domain with bounded ball deviation. Let $x,y\in \Omega$ with $x \ne y$ and $\cB_{x,y}$ be the ball in $\bbR^d$ centered at $\frac{x+y}{2}$ with radius $r_{x,y}=\frac{|x-y|}2$ and $\cR_{x,y}=\cB_{x,y}\cap\Omega$.  Let $\varpi_d$ be the  surface measure of the unit sphere in $\bbR^d$. Then for $f\in C^1(\Omega)$  
		\begin{equation} \label{eq:Poincare1} \Big |f(x)-\av_{\cR_{x,y}} f\Big | \leq \frac{2^d\dev(\Om)}{\varpi_d} %
			\int_{\cR_{x,y}} \frac{|\nabla f(z)|}{|x - z|^{d-1}} \d z.
		\end{equation}
		and 
		\begin{equation} \label{eq:Poincare2}
			|f(x)-f(y) |\le  \frac{2^d\dev(\Om)}{\varpi_d} %
			\Big( \int_{\cR_{x,y}} \frac{|\nabla f(z)|}{|x - z|^{d-1}} \d z + \int_{\cR_{x,y}} \frac{|\nabla f(z)|}{|y - z|^{d-1}} \d z\Big). 
		\end{equation} 
	\end{lemma}
	\begin{proof}
		To prove \eqref{eq:Poincare1} we fix $x$ and $y$ and let, for $\omega\in \bbS^{d-1}$, 
		$r_{\omega} = \sup\{r \geq 0 \colon x + r \omega \in \cR_{x,y}\}$. Then we have $r_{\omega} \leq |x-y| = \text{diam}(\cR_{x,y})$, and for every $s \in [0,r_{\omega})$, we have
		\begin{align*}
			|f(x) - f(x + s \omega)| &= \Big| \int_0^s \omega \cdot \nabla f(x + t \omega) \d t \Big| \\
			&\leq \int_0^{r_{\omega}} |\nabla f(x + t \omega) |\d t\,.
		\end{align*}
		Hence 
		\[
		\begin{split}
			|f(x) - \av_{\cR_{x,y}} f| &= \Big| f(x) - \frac{1}{|\cR_{x,y}|} \int_{\mathbb{S}^{d-1}} \int_0^{r_{\omega}} f(x + s \omega) s^{d-1} \d s \d \omega \Big| \\
			&\leq \frac{1}{|\cR_{x,y}|} \int_{\mathbb{S}^{d-1}} \int_0^{r_{\omega}} |f(x) - f(x + s \omega)| s^{d-1} \d s \d \omega \\
			&\leq \frac{1}{|\cR_{x,y}|} \int_{\mathbb{S}^{d-1}} \int_0^{r_{\omega}} \int_0^{r_{\omega}} |\nabla f(x + t \omega) |\d t \, s^{d-1} \d s \, \d \omega \\
			&= \frac{1}{|\cR_{x,y}|} \int_{\mathbb{S}^{d-1}} \frac{r_{\omega}^d}{d} \int_0^{r_{\omega}} |\nabla f(x + t \omega) |\d t \, \d \omega \\
			& \leq \frac{\text{diam}(\cR_{x,y})^d }{d |\cR_{x,y}|} \int_{\mathbb{S}^{d-1}} \int_0^{r_{\omega}} |\nabla f(x + t \omega) |\d t \, \d \omega. 
		\end{split}
		\]
		Since 
		\[ \int_{\cR_{x,y}} \frac{|\nabla f(z)|}{|z-x|^{d-1}} \d z= \int_{\mathbb{S}^{d-1}}\int_0^{r_{\omega}} |\nabla f(x + t \omega) |\d t \, \d \omega\] and since $\varpi_d/d$ is the $d$-dimensional volume of the unit ball,
		\[ \frac{\text{diam}(\cR_{x,y})^d }{d |\cR_{x,y}|}= \frac{2^d   \rad (\cB_{x,y})^d} {d|\cB_{x,y}\cap\Omega|}= \frac{2^d|\cB_{x,y}|} {\varpi_d |\cB_{x,y}\cap\Omega|} \le \frac{2^d \dev(\Om)}{\varpi_d}\] we obtain \eqref{eq:Poincare1}.
		
		Since $\cR_{x,y}=\cR_{y,x}$ we obtain from \eqref{eq:Poincare1}
		\begin{align*} 
			|f(x)-f(y)|&\le |f(x)-\av_{\cR_{x,y}} f|+ |f(y)-\av_{\cR_{y,x}} f|
			\\
			&\le \frac{2^d\dev(\Om)}{\varpi_d} %
			\Big( \int_{\cR_{x,y}} \frac{|\nabla f(z)|}{|x - z|^{d-1}} \d z + \int_{\cR_{x,y}} \frac{|\nabla f(z)|}{|y - z|^{d-1}} \d z\Big) 
		\end{align*} 
		which is \eqref{eq:Poincare2}. %
	\end{proof}
	
	\begin{lemma}\label{lem:diffqbound}
		Let $\Omega\subset \bbR^d$ be a convex domain with bounded ball deviation. Let $f \in C^1(\Om)$ and $x,y\in \Omega$. With $T_b$ as in \eqref{eq:Tgammadef} and $T^k_b$ as in \eqref{eq:Skgammadef}  we have \begin{multline} \label{eq:Qbound}
			\diff_{\gamma/\tempp}f(x,y) \lc_{d,\gamma,p} \dev(\Om)  \Big( \sum_{k \geq 2} 2^{-k} T_{\gamma/\tempp}^k |\nabla f|(x,y) + \sum_{k \geq 0} 2^{-k} 2^{-k \frac{\gamma}{\tempp}} T_{{\gamma}/{\tempp}} |\nabla f|(x,2^{-k}|y-x|)  \Big).
		\end{multline}
	\end{lemma}
	\begin{proof}
		
		We let $\cB_{x,y}$ and $\cR_{x,y}=\cB_{x,y}\cap \Omega$ be as in Lemma \ref{lem:poincare}. We dyadically decompose $\cB_{x,y}$ into the annuli
		\[ \cA^k_{x,y}= B(x, 2^{-k}|y-x|) \setminus B(x, 2^{-k-1}|y-x|) \]
		and estimate 
		\begin{align*} 
			\frac{|f(y)-f(x)|}{|y-x|} \lc_d \frac{\dev(\Om)}{|y-x|}   \sum_{k=0}^\infty 
			\Big( \int_{\cA^k_{y,x}\cap\Om}\frac{|\nabla f(z)| }{|y-z|^{d-1}} \d z + \int_{\cA^k_{x,y}\cap\cR_{x,y}}\frac{|\nabla f(z)| }{|x-z|^{d-1}} \d z\Big) 
		\end{align*} 
		and thus 
		
		\begin{equation} \label{eq:pointwise_Riesz}
			\begin{split}
				\frac{|f(y)-f(x)|}{|y-x|^{1+\frac{\gamma}{p}}} \lc_d \dev(\Om)   \Big(&\sum_{k \geq 0} 2^{-k}  \frac{2^{kd}}{|y-x|^{d+\frac{\gamma}{p}}} \int_{B(y,2^{-k}|y-x|) \cap \Om} |\nabla f(z)| \d z \\ +& \sum_{k \geq 0} 2^{-k} \frac{2^{kd} }{|y-x|^{d+\frac{\gamma}{p}}} \int_{B(x,2^{-k}|y-x|) \cap \cR_{x,y}} |\nabla f(z)| \d z\Big).
			\end{split}
		\end{equation}
		
		We obtain  with $r_k(x,y)= 2^{-k}|x-y|$, 
		\[\frac{2^{kd}}{|y-x|^{d+\frac{\gamma}{p}}} \int_{B(x,2^{-k}|y-x|) \cap \Om} |\nabla f(z)| \d z \lc_{d} 2^{-k\frac{\gamma} {p}} T_{\ga/p} [|\nabla f|] (x, r_k(x,y) ) .\]
		For $k=0,1, 2$, we bound
		\[\frac{2^{kd}}{|y-x|^{d+\frac{\gamma}{p}}}\int_{B(y,2^{-k}|y-x|) \cap \cR_{x,y}} |\nabla f(z)| \d z \lc_{d}   T_{\ga/p} [|\nabla f|] (x, |x-y| ).
		\]
		However  for larger $k$ we need to invoke the operators $T^k_{\ga/p}$ and  obtain 
		\[\frac{2^{kd}}{|y-x|^{d+\frac{\gamma}{p}}} \int_{B(y,2^{-k}|y-x|) \cap \cR_{x,y}} |\nabla f(z)| \d z \leq   T^k_{\ga/p} [|\nabla f|] (x, y),\,\, \text{ $k\ge 2$.}\]
		Combining the three previous displays with \eqref{eq:pointwise_Riesz} leads to \eqref{eq:Qbound}.
	\end{proof}

	To deal with the $k$-sums in \eqref{eq:Qbound}, we use the following lemma, which is a substitute for the failure of triangle inequality for the ${L^{p,\infty}}$ quasi-norm. Instead we could also use, for $p=1$   a summing lemma of Stein and N. Weiss \cite{SteinNWeiss} and for $p>1$ the equivalence of the seminorm to a norm \cite{Hunt-Ens}.

	\begin{lemma} \label{lem:sum_weakLp}
		Let  $1 \leq p < \infty$ and $ \eps>0$. Then there is $C=C_\eps>0$ independent of $p$ such that for any sequence of non-negative functions $G_k$ on the same measure space $(X,\mu)$, we have 
		\[
		\Big[ \sum_{k=1}^{\infty} 2^{-k\eps} G_k \Big]_{L^{p,\infty}(X,\mu)} \le C \sup_{k \geq 1} [G_k]_{L^{p,\infty}(X,\mu)}.
		\]
	\end{lemma}
	
	\begin{proof} %
		For $\lambda > 0$, since $1 = \sum_{k=1}^{\infty} \frac{6}{\pi^2} \frac{1}{k^2}$, we see that the set $\{x \in X \colon \sum_{k=1}^{\infty} 2^{-k\eps} G_k(x) > \lambda \}$ is contained in the union of the sets $\{x \in X \colon G_k(x) > \frac{6}{\pi^2} \frac{\lambda 2^{k\eps}}{k^2} \}.$ As $G_k\in L^{p,\infty}$ we get 
		for all $\lambda > 0$,
		\[\mu \Big\{x \in X \colon \sum_{k=1}^{\infty} 2^{-k \varepsilon} G_k(x) > \lambda \Big\}\le 
		\sum_{k=1}^{\infty} \big( \tfrac{\pi^2 k^2} {6\cdot  2^{k\eps}} \big)^p \lambda^{-p}   [G_k]_{L^{p,\infty}(X,\mu)}^p  \]
		and the assertion follows from  $\sum_{k=1}^{\infty} (\frac{k^2}{2^{k\eps}})^p \le  (\sum_{k=1}^\infty  \frac{k^2}{2^{k\eps}})^p,$  with $C= \frac{\pi^2}{6} \sum_{k=1}^\infty k^22^{-k\eps}$. 
	\end{proof}
	
	\begin{proof}[Proof of Theorem~\ref{thm:main}]
		Assume $1 \leq \tempq \leq \tempp < \infty$ and $(u,w) \in A_\tempq(\Om)$. We use Lemma \ref{lem:diffqbound} and let 
		$\varepsilon = 1 - \frac{(\tempq-1)d}{\tempp} > 0$, so that $2^{-k} = 2^{-k \varepsilon} 2^{-k\frac{(\tempq-1)d}{\tempp}}$. We estimate the first expression on the right hand side of \eqref{eq:Qbound} using  Lemma~\ref{lem:sum_weakLp} and Theorem~\ref{thm:g_weakLp-k}
		\begin{align*}
			\Big[ \sum_{k \geq 2} 2^{-k} T_{\gamma/\tempp}^k |\nabla f|(x,y) \Big]_{L^{\tempp,\infty}(\Om \times \Om, \mu_{\gamma}^u)}
			&\lc \sup_{k \geq 2} \Big[ 2^{-k\frac{(\tempq-1)d}{\tempp}} T_{\gamma/\tempp}^k |\nabla f|(x,y) \Big]_{L^{\tempp,\infty}(\Om \times \Om, \mu_{\gamma}^u)} \\
			&\lc_{d,\ga,p,q} [u,w]_{A_\tempq(\Om)}^{1/\tempp} \|\nabla f\|_{L^\tempp(\Om,w)},
		\end{align*}
		For the second expression on the right hand side of \eqref{eq:Qbound} we use Lemma~\ref{lem:sum_weakLp} again to estimate  
		\begin{align*}
			\Big[ \sum_{k \geq 0} &2^{-k} 2^{-k\gamma/\tempp} T_{\gamma/\tempp} |\nabla f|(x,2^{-k}|x-y|) \Big]_{L^{\tempp,\infty}(\Om \times \Om, \mu_{\gamma}^u)}
			\\&\lc \sup_{k \geq 0} \big[ 2^{-k\gamma/\tempp} T_{\gamma/\tempp} |\nabla f|(x,2^{-k}|x-y|) \big]_{L^{\tempp,\infty}(\Om \times \Om, \mu_{\gamma}^u)} 
		\end{align*} which is dominated by a constant times
		\[ \sup_{k\ge 0} \big[ 2^{-k\gamma/\tempp} T_{\gamma/\tempp} |\nabla f|(x,2^{-k} r) \big]_{L^{\tempp,\infty}(\Om \times (0,\infty), \nu_{\gamma}^u)} 
		= \big[ T_{\gamma/\tempp} |\nabla f|(x,r) \big]_{L^{\tempp,\infty}(\Om \times (0,\infty), \nu_{\gamma}^u)} .
		\] 
		(The last equality is by a change of variable formula that accounts for the scaling in $r$.)
		By Theorem~\ref{thm:g_weakLp}
		the last term  is $\lc_{d,\ga,p,q} [u,w]_{A_\tempq(\Om)}^{1/\tempp} \|\nabla f\|_{L^\tempp(\Om,w)}$ (we also used $[u,w]_{A_p(\Omega)} \leq [u,w]_{A_q(\Omega)}$ in the case $\ga > 0$).  This finishes the proof of Theorem~\ref{thm:main}.
	\end{proof}

	\section{Further discussion and perspectives} \label{sect:speculations}

	\subsection{Weighted Bourgain--Brezis--Mironescu formula}
	As $s\to 1^- $ the  fractional
	$\dot W^{s,p}$ seminorms of any non-constant $C^{\infty}$ function 
	tend to $\infty$. %
	Bourgain--Brezis--Mironescu \cite{Bourgain_Brezis_Mironescu_2001} who showed that  
	\begin{equation} \label{eq:BBM}
		\lim_{s \to 1^-} (1-s) \iint_{\R^d \times \R^d} \frac{|f(y)-f(x)|^p}{|y-x|^{s p + d}} \d x \d y = k(d,p) \int_{\R^d} |\nabla f(x)|^p \d x
	\end{equation}
	for some constant $k(d,p)$ depending only on $d$ and $p$. Thus the blowup of the double integral as $s\to 1^-$ can be fixed by multiplying it with a factor $1-s$.
	This   formula has been extended in many directions (see e.g. \cites{Bojarski-bbm, MR4525722, DrelichmanDuran2022, HMPV, MR2773218}), and more recently even with pointwise  analogues using  Spector's nonlinear ``fractional gradient" operator %
	$$\mathscr D^s f(x) =\int_{\R^d}\frac{|f(x)-f(y)|}{|x-y|^{d+s}} \d y,$$
	from \cite{Spector20}. 
	Namely, letting $I_s$ denote the Riesz potential operator, defined for $0<s<d$ by 
	$$I_{s} f(x)=\int_{\mathbb{R}^d} |x-y|^{s-d} f(y) \d y,$$ 
	there is the pointwise limit
	\begin{equation*}
		\lim_{s \to 1^-}  (1-s)I_{s}(\mathscr{D}_{s}f)(x) = k(d,1) \, I_1(|\nabla f|)(x),
	\end{equation*}
	which was recently established in \cite{CP2026}.

	Theorem \ref{thm:main} is partially motivated by 
	a local weighted Bourgain--Brezis--Mironescu inequality for the case $p=1$, with 
	$A_1$  pairs of weights of the form $(\mu, M\mu)$, which  can be found in \cite[Theorem 2.1]{HMPV}. Here
	we let $Q$ be a cube and let $\nu$ be a Radon measure on $\mathbb{R}^d$, $d\ge 1$. Then, for every $0<s<1$ and $1\le p<\infty$, there exists a dimensional constant $c_d>0$ such that
	\begin{equation}\label{modelExample}
		\int_Q \int_Q
		\frac{|f(x)-f(y)|^p}{|x-y|^{d+sp}}
		\d y\d \nu(x)
		\le
		\frac{c_d}{1-s}
		(\sqrt d,\ell(Q))^{(1-s)p}
		\int_Q
		|\nabla f(x)|^p
		\,M(\chi_Q\nu)(x)
		\d x
	\end{equation}
	for every $f\in C^1(Q)$. 
	If we omit the factor $(1-s)^{-1}$ the corresponding inequality for $s=1$ tends to be false for non-constant $f$ (recall the already mentioned case whenever   $\nu$ is Lebesgue measure, see \cite{brezisconstants}).  
	(See also the previous work \cite[Theorem 6.2]{MR4789312}   with the  constant $s^{-1}(1-s)^{-1} $ in front).

	It should   be interesting to further pursue the study of such weighted estimates with the correct $(1-s)^{-1}$ blowup,  and for  general pairs of weights $(u,w)$ in the classes $A_q(Q)$.

	\subsection{Difference quotients and maximal functions}\label{sec:A1disc} Historically, unweighted estimates for difference quotients have been studied, in the case $1 < p < \infty$, via the classical estimate
	\begin{equation} \label{LipscType}
		\frac{|f(x)-f(y)|}{|x-y|} \lesssim_d \,\big (M |\nabla f|(x)+M|\nabla f|(y) \big)
	\end{equation}
	which seems to be written down first in  Bojarski \cite{Bojarski90}. It is valid at least for  all $f \in C^1(\R^d)$, $x\neq y$,  and can be  found in  \cite{KLV}. 
	P. Hajłasz used this estimate in \cite{Hajlasz1996} as the key tool to define Sobolev spaces $W^{1,p}$ on a metric space with a Borel measure for $p\in(1,\infty)$.  Its various uses have been discussed for example in \cite{Nguyen06}, \cite{CrippaDeLellis}, \cite{BVY} and \cite{BSVY}. %
	In one dimension, only one of the terms on the right hand side of \eqref{LipscType} is present (which is immediate from the fundamental theorem of calculus). In two and higher dimensions  arguments in \cites{Nguyen06, BVY, BSVY}  
	via \eqref{LipscType} (relying on the method of rotations in the unweighted case)  are insufficient in the weighted setting and do not recover  Theorem \ref{thm:main}. 
	For illustration we consider the case  of $\Omega=\bbR^d$.

	\begin{prop}\label{prop:M2}
		For $d \geq 1$, $\gamma\neq 0$, $1<p<\infty$,
		\begin{equation} \label{M2ineq} 
			\mu^u_\gamma(\{ (x,y): \diff_{\ga/p} f(x,y)>\la \}) \lc_{d,\gamma,p} \la^{-p} \int_{\R^d} |\nabla f(x)|^p M^2u(x) \d x
		\end{equation}	 for all $f\in C^1(\bbR^d)$. 
		In dimension $d=1$,  $M^2u$ on the right hand side of \eqref{M2ineq} can be replaced with $Mu$.             \end{prop} 
	Note that 
	$Mu(x) \leq M^2u(x)$ for almost every $x \in \R^d$.  In light of the case $d=1$ 
	and Corollary \ref{cor:A1} 
	it would be desirable to try to prove a stronger inequality with $M^2u$ replaced by $Mu$ also in dimension $d\ge 2$.

	\begin{proof}
		Using  \eqref{LipscType}, we have
		\begin{equation} \label{eq:split2M}
			[\diff_{\ga/p} f]
			_{L^{p,\infty}(\d \mu^u_\ga)} 
			\lesssim_d \, \Big[\frac{M|\nabla f|(x)}{|y-x|^{\frac{\ga}{p}}} \Big]_{L^{p,\infty} (\d \mu^u_\ga)} + \Big[\frac{M|\nabla f|(y)}{|y-x|^{\frac{\ga}{p}}} \Big]_{L^{p,\infty}(\d \mu^u_\ga)}.
		\end{equation}
		
		For the first term on the right hand side of 
		\eqref{eq:split2M}
		we have a good bound, involving only one maximal function on the right hand side, namely  we can prove
		\begin{equation} \label{M2ineq1}
			\Big[\frac{M |\nabla f|(x)}{|y-x|^{\frac{\ga}{p}}} \Big]_{L^{p,\infty}(\bbR^d\times\bbR^d,  \d \mu^u_\ga )} \lesssim_{\ga,p}\|\nabla f\|_{L^p(\R^d, Mu(x) \d x)}. 
		\end{equation}	
		To see this let 
		$ \cE_1= \{(x,y) \in \R^d \times \R^d \colon |x-y|^{-\ga/p} {M|\nabla f|(x)}
		> \la  \}.$
		Note that for each fixed $x\in \bbR^d$
		\[ \int_{\{y: |x-y|^{-\ga/p} Mf(x)>\la \}} |x-y|^{\gamma-d} \d y\lc_{\ga,d} \Big( \frac{Mf(x)}{\la}\Big)^p. \]
		This works by direct calculation for both positive and negative $\gamma$ (but not for $\gamma=0$). 
		Integrating in $x$ we obtain 
		$$\mu^u_\gamma (\cE_1) \lc_{d,\gamma}  \la^{-p} \int M(|\nabla f|)(x)^p u(x) \d x
		$$
		and we can invoke the well-known Fefferman--Stein inequality \cite{FS}:  if $1<p<\infty$  there exists a constant $C_d$ such that for all non-negative measurable functions $g$ on $\R^d$,
		\begin{equation}\label{FS}
			\int_{\R^d} Mg(x)^p \,  u(x) \d x\leq C_d\,p'\int_{\R^d} |g(y)|^p\,Mu(y)\d y.
		\end{equation}
		Applying this to $g = |\nabla f|$, we obtain
		$$
		\int_{\cE_1} u(x) \d x \d y \lesssim_d \,p'  \la^{-p} \int_{y\in \R^d} |\nabla f|(y)^p\,Mu(y)  \d y,
		$$
		proving \eqref{M2ineq1}.
		
		However,  for the second term on the right hand side of \eqref{eq:split2M}  which is only present in dimensions $d\ge 2$, we can only prove 
		\begin{equation} \label{M2ineq2}
			\Big[\frac{M |\nabla f|(y)}{|y-x|^{\frac{\ga }{p}}} \Big]_{L^{p,\infty}(\R^d \times \R^d, \mu^u_\gamma )} \lesssim_{d,p} \|\nabla f\|_{L^p(\R^d, 	M^2u)}
		\end{equation}
		with two maximal functions on the right hand side, which gives rise to the two maximal functions on the right hand side of  \eqref{M2ineq}.	
		Let $
		\cE_2= \{(x,y) \in \R^d \times \R^d \colon |x-y|^{-\ga/p} {M|\nabla f|(y)}	   > \la  \}$.
		\eqref{M2ineq2} will follow if we can prove			\begin{equation}\label{eq:E2est}\mu^u_\ga(\cE_2) = \int_{y\in \bbR^d} \int_{x: |x-y|^{-\ga/p} M|\nabla f|(y) >\la } |x-y|^{\gamma-d}u(x) \d x \d y
			\lesssim_{d,p} \la^{-p} \int| \nabla f(y)|^p M^2u(y) \d y.
		\end{equation} 
		For fixed $y$ we estimate the inner integral.
		Assuming $M|\nabla f|(y)<\infty$ we  let,  for $n\in \bbN$,
		\[\cA_{n,\ga}(y)= \big\{ x\in \bbR^d:
		\la 2^{n-1}<|x-y|^{-\ga/p} M(|\nabla f|)(y) \le \la 2^n \big\}.\]      
		
		Then, for both positive and negative $\gamma$, 
		\begin{align*} \int_{x\in \cA_{n,\ga}(y)}|x-y|^{\gamma-d} u(x) \d x&\lc\Big(  \frac{M(|\nabla f|)(y)}{\la 2^n}\Big) ^{\frac {p}{\gamma}(\ga-d)}  %
			\int_{|x-y|\lc_{d,\ga}(
				\frac{M(|\nabla f|)(y)}{\la 2^n})^{\frac {p}{\gamma}   }} u(x) \d x
			\\ &\lc 2^{-np} \Big(  \frac{M(|\nabla f|)(y)}{\la }\Big) ^{p} M u(y). 
		\end{align*} 
		We integrate in $y$ and sum in $n$ to  obtain 
		$$\mu^u_\gamma(\cE_2) \lc_{d} 
		\int M(|\nabla f|)(y)^p Mu(y) \d y$$
		and another application of the Fefferman-Stein inequality \eqref{FS} yields \eqref{eq:E2est} and then \eqref{M2ineq2}.			\end{proof}

	\begin{remarka}
		Let $\nu$ be a nonnegative Radon  measure on $\bbR^d$.  
		In Proposition \ref{prop:M2} the pairs $(u, Mu)$ and  $(u,M^2 u)$  may also be replaced by
		$(\nu, M\nu)$ and $(\nu, M^2\nu)$, respectively,  if we replace the measure $\mu^u_\gamma$ with 
		$|x-y|^{\ga-d} \d\nu(x)$. One uses here that the classical Fefferman--Stein inequality \cite{FS}  also holds for general measures $\nu$.   
	\end{remarka}

	There are various refinements of  \eqref{LipscType}, which are or may  be  useful  in applications.
	One involves more regular maximal functions (leading for example to a version of Bressan's problem with Hardy spaces, see \cite[\S5.1]{HSSS}). 
	Another potentially useful replacement of \eqref{LipscType} (which had been
	included in a preliminary unpublished preprint  version of \cite{HMPV}, arXiv 2302.14029) 
	is contained in the following proposition.

	\begin{prop}\label{prop:newLip}   For all $f \in C^1(\R^d)$, all cubes $Q \subset \R^d$, and for a.e. $(x, y) \in Q\times Q$,
		\begin{equation} \label{new-LipscType}
			\frac{| f(x)-f(y) |}{|x-y|}  \lesssim_d   \min_{z\in\{x,y\}} M(|\nabla f|\chi_{Q})(z)^{\frac{1}{d}}\max_{z\in\{x,y\}} M(|\nabla f|\chi_{Q})(z)^{1-\frac{1}{d}}.\\
		\end{equation}
	\end{prop}
	\eqref{LipscType}  follows from \eqref{new-LipscType} by Young's inequality, i.e. \( ab\leq a^p/p+ b^{p'}/p'\) for  \( a, b \geq 0 \), and \( p > 1 \). Other  extensions  of \eqref{LipscType} in different directions will be included in the upcoming work  \cite{GLP}.

	For the proof of Proposition \ref{prop:newLip} we consider  again, for  $0<s<d$, the standard Riesz potential of a non-negative measure $\mu$, defined by 
	$I_{s} \mu(x)=\int_{\R^d} |x-y|^{s-d}  \d \mu(y) $. 
	We have the following lemma: 
	
	\begin{lemma}\label{New-pointwise-B}
		Let $Q$ be a cube in $\R^d$, $\mu$ be a non-negative measure and $0<s<d$. Then for almost every $x\in Q$,
		\begin{equation} \label{eq:Hedb-meas}
			I_s(\mu \chi_{Q})(x) \leq \frac{c_d}{s} \mu(Q)^\frac{s}{d}M(\mu\chi_{Q})(x)^\frac{d-s}{d} \,.
		\end{equation}
		In particular,
		\begin{equation} \label{eq:Hedb}
			I_s(\mu \chi_{Q})(x) \leq \frac{c_d}{s}  \ell(Q)^s\, M(\mu\chi_{Q})(x).
		\end{equation}
	\end{lemma}
	The second estimate  \eqref{eq:Hedb} 
	is a version of  Hedberg's inequality \cite{Hedberg1972}.  However, \eqref{eq:Hedb-meas} 
	is sharper and it is more relevant for proving \eqref{new-LipscType}. 
	
	\begin{proof}[Proof of Lemma \ref{New-pointwise-B}]
		We may assume that $\mu(Q)\neq 0$, otherwise there is nothing to prove. 
		Let $Q\subset\mathbb{R}^d$ be a fixed cube and $x\in Q$. 
		To prove \eqref{eq:Hedb-meas}  we let $t>0$ and define 
		\(
		Q_{x,t}
		\)
		as the cube with center at $x$ and side length $2t^{-\frac{1}{d-s}}$. 
		Then using the layer cake formula, we obtain
		\begin{align*}
			\int_{Q}\frac{\d\mu(y)}{\lvert x-y \rvert^{d-s}}
			&= 
			\int_0^{\infty} \mu\Big(\Bigl\{y\in Q: \frac{1}{\lvert x-y \rvert^{d-s}}>t\Bigr\}\Big)  \d t
			= \int_0^{\infty} \mu\bigl(\bigl\{ y\in Q: \lvert x-y \rvert<t^{-\frac{1}{d-s}} \bigr\}\bigr) \d t \\
			&\leq \int_0^{\infty} \min\Big\{\mu(Q),
			\frac{\mu(Q_{x,t})}{|Q_{x,t}|} |Q_{x,t}| \Big\} \d t 
			\leq \int_0^{\infty}\min\bigl\{\mu(Q),M(\mu \chi_Q)(x) 2^dt^{-\frac{d}{d-s}}\bigr\} \d t
			\\
			&=\int_0^{2^{d-s}(M(\mu \chi_Q)(x)/\mu(Q))^\frac{d-s}{d}}\mu(Q)
			\d t
			+ 2^d \int_{2^{d-s}(M(\mu \chi_Q)(x)/\mu(Q))^\frac{d-s}{d}}^\infty M(\mu \chi_Q)(x)t^{-\frac{d}{d-s}}
			\d t
			\\
			&=
			\frac{2^{d-s}d}{s}
			\mu(Q)^\frac{s}{d}(M(\mu \chi_Q)(x))^\frac{d-s}{d}. \qedhere
		\end{align*}
	\end{proof}

	\begin{proof}[Proof of Proposition \ref{prop:newLip}]
		For $d=1$ the proof is immediate from
		$|f(x)-f(y)|\leq \int_x^y|f'(t)|dt\leq |x-y|M(|f'|\chi_Q)(x)$. In what follows we assume $d\ge 2$. 
		
		Let $Q$ be a cube in $\R^d$ and $f \in W^{1,1}(Q)$. For any cube $Q \subset \R^d$ and any Lebesgue points $x, y$ of $f$ in $Q$, there is a closed cube $R \subset Q$ containing both $x$ and $y$, for which $\ell(R) \leq |x-y|$.
		We will use the well-known subrepresentation formula   
		\begin{equation*}%
			|f(x)-\av_R f|\leq c_d 
			I_{1}(|\nabla f| \chi_{R})(x)
		\end{equation*}
		(see e.g. Lemma 7.16 of \cite{03561752}).
		Hence by \eqref{eq:Hedb-meas}
		\begin{equation*}%
			\lvert f(x)-\av_R f\rvert \leq
			c_d \, \Big(\int_{R}|\nabla f(z)|\d z\Big)^\frac{1}{d}M(|\nabla f| \chi_{R})(x)^{1-\frac{1}{d}} .
		\end{equation*}
		If we apply this formula with $y$ in place of $x$ and 
		observe that 
		$\lvert f(x)-f(y)\rvert$ is dominated by $\lvert f(x)-\av_R f \rvert +\lvert f(y)-\av_R f\rvert $ we obtain 
		\[
		\lvert f(x)-f(y)\rvert  \leq 2 c_d \, \Big(\int_{R}|\nabla f|(z)\,\d z\Big)^\frac{1}{d} \max_{z \in \{x,y\}} M(|\nabla f| \chi_{Q})(z)^{1-\frac{1}{d}}.
		\]
		Furthermore, since $x, y \in R$, 
		\[
		\int_{R}|\nabla f(z)|\,\d z \leq \int_{B(x,\ell(R))} |\nabla f(z)| \chi_Q(z) \, \d z \lesssim_d \ell(R)^d M(|\nabla f| \chi_Q)(x)
		\]
		and similarly  the same estimate with $x$ replaced by $y$.
		Thus 
		\[
		\Big( \int_{R}|\nabla f(z)|\,\d z \Big )^{\frac{1}{d}} \lesssim_d \ell(R) \min_{z \in \{x,y\}} M(|\nabla f| \chi_Q)(z)^{\frac{1}{d}}
		\]
		and \eqref{new-LipscType} follows since $\ell(R) \leq |x-y|$.
	\end{proof}
	
	\subsection{More results for $d=1$}  The following elementary results on $\diff_{\ga/p}$ for $d=1$, $\ga<0$  have been assisted by  ChatGPT5.5-Pro.
	Again, they raise corresponding  questions in higher dimensions. 
	
	\subsubsection{Failure of a two-weight  $A_p$  result for $\gamma=-p$.}\label{sec:example}  We report  a necessary condition  in the two weight case which does not seem to have an analogue in the one weight theory. Let $d=1$,  
	$p=q\ge 1$, $\gamma=-p$, then  inequality   \eqref{eq:p-p}    fails for some $(u,w)\in A_p(\bbR)$. 
	Namely, let $u(x)=(1+|x|)^{p-1}$, $w(y)= (1+|y|)^\beta$ for $\beta>  p-1$; then $(u,w) \in A_p(\bbR)$ (note that this fails for $\beta=p-1$ when $p > 1$). Let $f\in C^\infty(\bbR)$ such  that $f(y)=0$ for $y<-1$ and $f(y)=1$ for $y>1$ so that $\int|f'|^pw<\infty$. For $\la<1$ the set $E_{\la,-1}$ contains $A_\la=\{(x,y): x<-1, y>1\}$ which  satisfies  $\mu^u_{-p}(A_\la)=\infty$ for $\la<1$. 
	
	\subsubsection{A two-weight  $A_q$ inequality  in one dimension} In Theorem \ref{thm:main} the range of $\gamma$ can be improved for $d=1$ and $q<p$. Namely for $\gamma<0$, 
	\begin{equation} \label{eq:Aq1d}
		\mu_\gamma^u(E_{\lambda, \ga/p}(f)) \lc [u,w]_{A_q(\bbR)} \la^{-p} \int_{\R} |f'(y)|^p w(y) \d y, \quad 1\le q<p<\infty .
	\end{equation} We argue as in \S\ref{sec:A1disc} above (see also  \cites{Nguyen06, BSVY}). Since   
	$|f(x)-f(y)|\le |x-y| M[f'](x)$ for $d=1$ we have  $\fD_{\ga/p} f(x,y)\le |x-y|^{-\ga/p} M[f'](x)$. If $\fD_{\ga/p} f(x,y)>\la$ then $|x-y|\ge \la^{-p/\gamma} M[f'](x)^{p/\gamma}$; hence
	\[\mu_\gamma^u(E_{\lambda, \ga/p}(f))  \le \int_{\R} u(x) \int_{\substack {\{y:\,\, |x-y| \ge \la^{-p/\ga}M[f'] (x)^{p/\ga} \} }}|x-y|^{\ga-1} \d y\,\d x \lc \la^{-p} \int_{\R} |M[f'](x)|^p u(x) \d x.\] By  the weak type inequality \eqref{charactAp-two-weights}, with  $q$ in place of $p$, and subsequent    Marcinkiewicz interpolation with the $L^\infty$ bound we get 
	$\| M g \|_{L^p(u)} \lc_{p,q} [u,w]_{A_q}^{1/p}  \|g\|_{L^p(w)}$    for $q<p<\infty$,
	and \eqref{eq:Aq1d} follows.

	\section*{Acknowledgements}
	This work was initiated 
	during a visit of   
	P\'{e}rez 
	to the Mathematical Sciences Institute at the Australian National University, whose support we gratefully acknowledge.  
	P\'{e}rez 
	is partially supported by the Spanish government through the grant PID2023-146646NB-I00 and by
	Severo Ochoa accreditation CEX2021-001142-S, both at BCAM, and also by the Basque Government
	through grant IT1615-22 at the University of the Basque Country and by the BERC programme 2022-2025 at BCAM. 
	Seeger 
	was supported in part by NSF
	Grant 2348797, and Yung   supported by a Discovery Project DP250103744 from the Australian Research Council.

\end{document}